\documentclass{amsart}
\usepackage{CJKutf8}%\usepackage{amsmath,amsthm,amssymb,amsfonts,xypic}
\usepackage{amsmath,amsfonts,amssymb,amscd,bm,latexsym}
\usepackage[colorlinks=true]{hyperref}
\usepackage{tikz}

\usepackage{pgflibraryarrows}
\usepackage{pgflibrarysnakes}
\usepackage{pgflibraryarrows}
\usepackage{pgflibrarysnakes}
\hypersetup{urlcolor=blue, citecolor=blue}%red}
\numberwithin{equation}{section} %To use the equation number such as (1.1) etc
\def\<{\langle}             \def\>{\rangle}

\newtheorem{thm}{Theorem}[section]

\newtheorem{lem}[thm]{Lemma}

\theoremstyle{definition}

\newtheorem{rem}{Remark}[section]

\newcommand{\beeq}{\begin{equation}}\newcommand{\eneq}{\end{equation}}
\newcommand{\al}{\alpha} 
\newcommand{\gt}{\gtrsim}

   \newcommand{\be}{\beta}
    
\newcommand{\ep}{\varepsilon}
  
    \newcommand{\la}{\lambda}
    
\newcommand{\om}{\omega}    
    \newcommand{\Ga}{\Gamma}
\newcommand{\R}{\mathbb{R}}
\newcommand{\N}{\mathbb{N}}
\newcommand{\Sp}{\mathbb{S}}
\newenvironment{prf}{\noindent {\bf Proof.} }{\endprf\par}
\def \endprf{\hfill  {\vrule height6pt width6pt depth0pt}\medskip}
\newcommand{\pa}{\partial}
\newcommand{\les}{{\lesssim}}
\newcommand{\supp}{\,\mathop{\!\mathrm{supp}}}

\newcommand{\gm}{\mathfrak{g}}

\numberwithin{equation}{section}
\title[Finite time blow up for systems of wave equations with multiple speeds]
      {Finite time blow up for systems of wave equations with multiple speeds posed on asymptotically Euclidean manifold}
\author{Mengyun Liu}
\address{Department of Mathematics\\                	
Zhejiang Sci-Tech University\\                Hangzhou 310018, P. R. China}
\email{mengyunliu@zstu.edu.cn}

\thanks{The author was supported by NSFC 12101558 and NSF of Zhejiang province LQ22A010016}
\date{\today}
\dedicatory{} \commby{}
\begin{document}
%\begin{CJK}{UTF8}{gkai}
\begin{abstract}
In this work, we study the finite time blow-up phenomenon of three types of semilinear wave systems with multiple speeds, posed on asymptotically Euclidean manifolds. We establish the upper bound estimates for the lifespan of solutions when the spatial dimension $n\geq 2$. In particular, for system related to the Glassey conjecture, we obtained the finite time blow up results under a new curve. This new curve is sharp at least for $n=3$. 
 \end{abstract}



\maketitle

\section{Introduction}
In this paper, we study the finite time blow up phenomenon of three kinds of semilinear wave systems with multiple speeds, posed on asymptotically Euclidean (Riemannian) manifold $(\R^{n}, \gm)$ with $n\ge 2$. Before preceding, we give the definition of asymptotically Euclidean. We mean that
$(\R^n, \gm)$ is certain perturbation of the Euclidean space $(\R^n, \gm_0)$. More precisely, we assume
$\gm$ can be decomposed as 
\begin{align}
\label{adl1}
\gm=\gm_{1}+\gm_{2}\ ,
\end{align}
where $\gm_{1}$ is a spherical symmetric, long range perturbation of $\gm_0$, and $\gm_{2}$ is an exponential (short range) perturbation.
By definition, there exists polar coordinates $(r,\omega)$ for $(\R^n,\gm_1)$, in which we can write
\beeq
\label{adl3}
\gm_{1}=K^{2}(r)dr^{2}+r^{2}d\omega^{2}\ , 
\eneq
where $d\om^2$ is the standard metric on the unit sphere $\Sp^{n-1}$, and
\beeq
\label{adl2}
|\pa^{m}_{r}(K-1)|\les  \<r\>^{-m-\rho}
,  m=0, 1, 2.
\eneq
for some given constant $\rho >0$. 
Here and in what follows, 
$\langle x\rangle=\sqrt{1+|x|^2}$, and
we use $A\les B$ ($A\gtrsim B$) to stand for $A\leq CB$ ($A\geq CB$) where the constant $C$ may change from line to line. 
Equipped with the coordinates $x=r\omega$, we have
$$\gm=g_{jk}(x)dx^j dx^k\equiv \sum^{n}_{j,k=1}g_{jk}(x)dx^j dx^k\ ,\ \gm_2=g_{2, jk}(x)dx^j dx^k\ ,$$
where we have used the convention that Latin indices $j$, $k$ range from $1$ to $n$ and the Einstein summation convention for repeated upper and lower indices. 
Concerning $\gm_2$, we assume
it is an exponential (short range) perturbation of $\gm_1$, that is, there exists $\be>0$ so that
\beeq
\label{dlfjia}
|\nabla g_{2,jk}|+|g_{2,jk}|\les e^{-\be\int^{r}_{0}K(\tau)d\tau},\ 
|\nabla^2 g_{2,jk}|\les 1\ .
\eneq
By asymptotically Euclidean and Riemannian assumption, it is clear that 
there exists a constant $\delta_{0}\in(0, 1)$ such that 
\beeq
\label{unelp}
\delta_{0}|\xi|^2\le
g^{jk}\xi_{j}\xi_{k}\le
\delta_{0}^{-1} |\xi|^2, \forall \ \xi\in\R^{n},\ 
K\in (\delta_{0}, 1/\delta_{0})
\ .
\eneq

 The first kind of system that we studied is semilinear wave equation related to the Strauss conjecture. More precisely,  we consider the following system with multiple speeds posed on asymptotically Euclidean manifolds \eqref{adl1}-\eqref{unelp}.
\beeq
\label{2}
\begin{cases}
\pa^{2}_t u-c^2\Delta_{\gm} u=|v|^p \ , (t, x)\in (0, T)\times \R^n\ ,\\
\pa^{2}_t v-\Delta_{\gm} v=|u|^q\ ,\\
u(0, x)=\ep u_0(x), \ u_t(0, x)=\ep u_1(x)\\
v(0, x)=\ep v_0(x), \ v_t(0, x)=\ep v_1(x)\\
\end{cases}
\eneq
Here
$1<p,q\in\R$, $0<c\in \R$, $(u_0, u_1), (v_0, v_1)\in C^\infty_c(\R^n)$ and the small parameter $\ep>0$ measures the size of the data. As usual, to show blow up, we assume  both $(u_0, u_1)$ and $(v_0, v_1)$ are nontrivial, nonnegative and supported in $B_R:= \{x\in\R^n: |x|=r\le R\}$ for some $R>0$. It has been well investigated that the critical exponet $p_S(n)$ to Strauss problem
$$\pa^{2}_t u-\Delta u=|u|^p\ ,$$
is the positive root of the equation
$$(n-1)p^{2}-(n+1)p-2=0
\ .$$
As a natural generalization, there is interest in determining whether a critical curve exists for the corresponding system \eqref{2}. 

When $c=1$ and $\gm=\gm_0$, the critical curve has been well studied by Santo-Georgiev-Mitidieri \cite{SGM}. They showed that the curve
$$\Ga_{SS}(p,q, n):=\max\{\frac{p+2+q^{-1}}{pq-1}\ ,\frac{q+2+p^{-1}}{pq-1}\}-\frac{n-1}{2}=0\ , n\geq 2\ ,$$
plays the same role of $p_S(n)$ to Strauss problem at least in a neighborhood of $(p_S(n), p_S(n))$. More precisely, they obtained the finite time blow up results when $\Ga_{SS}(p,q,n)>0$, $n\geq 1$ and the global existence when $\Ga_{SS}(p, q, n)<0\ ,n\geq 2$ under some additional restriction on $p, q$. See also \cite{S} for case $n=3$ and blow up result in \cite{Deng}. For the critical case $\Ga_{SS}(p,q,n)=0$, it has also been proved there is no global existence in low space dimensions. See \cite{SM}, \cite{AKT}, \cite{KO} \cite{KO1} for $n=3$ and  \cite{KO1} for $n=2$. %has obtained the blow up region of system \eqref{2} in 
%$$1\leq p<p_S(n)$$
%$$1<pq<\frac{(n+3)p+2}{(n-1)p-2}\ , n\geq 4 \ .$$
Recently, Ikeda-Sobojima-Wakasa \cite{I-S-W} obtained upper bound estimates of the lifespan when
$\Ga_{SS}(p, q, n)\geq 0\ , n\geq 2\ $.

When the propagation speeds are different, that is $c\neq 1$, the existing results are all in low dimensional Euclidean space $\gm=\gm_0$. If $p\leq q$, it is easy to see that
$$\Ga_{SS}(p, q, n)=\frac{q+2+p^{-1}}{pq-1}-\frac{n-1}{2}=0\ ,$$
in which case Kubo-Ohta \cite{KO}, \cite{KO2} proved that the critical curve is still 
$\Ga_{SS}(p, q, n)=0$ in low dimensions $n=2, 3$. More precisly, they obtained the global existence when $\Ga_{SS}(p, q, n)<0$ and the finite time blow up results when $\Ga_{SS}(p,q, n)\geq 0$. To the best of the authors' knowledge, there is no results for higher dimensions. In this paper, we obtain the finite time blow up results with $\Ga_{SS}(p,q,n)>0, n\geq 2$ as well as upper bound estimates of the lifespan, to system \eqref{2}, for any $c>0$, posed on asymptotically Euclidean manifolds.

For the strategy of proof, we turn the systems into ODE systems by defining suitable functionals. For those ODE systems, we apply iteration methods to give upper bound estimates of the lifespan. The main innovation  is 
 the existence of  a class of  generalized  ``eigenfunctions" for 
 the Laplace-Beltrami operator,
 $\Delta_{\gm}\phi_\la=\lambda^{2}\phi_\la$, with small parameter $\lambda\in (0,\la_0)$ and desired asymptotical behavior, see Lemma \ref{elp}. As usual, to show blow up, we need the solution to satisfy finite speed of propagation. For the linear equation $\pa^{2}_{t}u-\Delta_{\gm_1}u=0$, with compacted initial data, the support of solution $u$ satisfies 
\beeq
\label{kl1}
\supp u\subset \{(t, x); \int^{|x|}_{0}K(\tau)d\tau\leq t+R_{1}\}\ .
\eneq
As $\gm_2$ is short-range perturbation, which does not affect speed of propagation too much, we still have \eqref{kl1} to $\pa^{2}_t u-\Delta_{\gm} u=0$, for some $ R_{1}\geq\int_{0}^{R_{0}}K(\tau)d\tau$. Hence we assume the support of the solutions to system \eqref{2} satisfies
 \begin{align}
\label{fspyl}
\begin{cases}
\supp u\subset D_1=: \{(t, x); \int^{|x|}_{0}K(\tau)d\tau\leq ct+R_{1}\}\ ,\\
\supp v\subset D_2=: \{(t, x); \int^{|x|}_{0}K(\tau)d\tau\leq t+R_{1}\}\ .
\end{cases}
\end{align}

\begin{thm}
\label{thm4}
Let $n\geq 2$, $0<c\in \R$ and $\Ga_{SS}(p, q, n)> 0$. Consider the system \eqref{2} posed on asymptotically Euclidean manifolds \eqref{adl1}-\eqref{unelp}.
 Assuming that
the data are nontrivial and satisfying 
 \begin{align}
 \label{73}
 \begin{split}
 \begin{cases}
 \int u_i dv_{\gm}>0\ ,  \int_{\R^n} u_0\phi_{\la}(x)d v_{\gm}>0\ , \int_{\R^n} (u_1-c\la u_0)\phi_{\la}(x)d v_{\gm}>0\ ,\\
\int v_i dv_{\gm}>0\ , \int_{\R^n} v_0\phi_{\la}(x)d v_{\gm}>0\ , \int_{\R^n} (v_1-\la v_0)\phi_{\la}(x)d v_{\gm}>0\ , 
\end{cases}
\end{split}
 \end{align}
 where $i=0, 1$. 
 Suppose it has a weak solution $u, v\in C^{2}([0, T_\ep);\mathcal{D}'(\R^{n}))$ with  $|v|^{p}, |u|^{q}\in C([0, T_\ep);\mathcal{D}'(\R^{n}))$ and satisfying \eqref{fspyl}. Then 
there exist a constant $\ep_0>0$, such that for any $\ep\in (0, \ep_{0})$, the solution will blow up at finite time. 
Moreover, we have the following estimate for the lifespan $T_\ep$
\begin{align*}
T_{\ep} \ \les \ 
\ep^{-\Ga_{SS}(p,q,n)^{-1}}\  .
\end{align*}
\end{thm}
\begin{rem}
Kubo-Ohta \cite{KO}, \cite{KO2} proved that the critical curve is 
$\Ga_{SS}(p, q, n)=0$ in low dimensions $n=2, 3$ whenever $0<c\in\R$. Based on our result, it is very natural to conjecture that for any $c> 0$ and $n\geq 2$, the critical curve to \eqref{2} is $\Ga_{SS}(p, q, n)=0$. 
\end{rem} 
 The second kind of system that we considered is closely related to the Glassey conjecture. 
 \beeq
\label{1}
\begin{cases}
\pa^2_t u-c^2\Delta_{\gm} u=|v_t|^p \ , (t, x)\in (0, T)\times \R^n\ ,\\
\pa^2_t v-\Delta_{\gm} v=|u_t|^q\ ,\\
u(0, x)=\ep u_0(x), \ u_t(0, x)=\ep u_1(x)\\
v(0, x)=\ep v_0(x), \ v_t(0, x)=\ep v_1(x)\\
\end{cases}
\eneq
Similarly, we assume $(u_0, u_1), (v_0, v_1)\in C^\infty_c(\R^n)$ and they are nontrivial, nonnegative and supported in $B_R$. It well known that the critical exponent to  
$$\pa^{2}_t u- \Delta u=|u_t|^{p}\ , $$
is $p_G(n)=1+2/(n-1)$. It is also interesting to determine the critical curve of the corresponding system \eqref{1}. 

When $c=1$ and $\gm=\gm_0$, Deng \cite{Deng} has obtained the blow up results of system \eqref{1} under the condition
\beeq
\nonumber
\label{region1}
\begin{split}
1<pq<\infty\ ,  
\begin{cases}
1\leq p<\infty\ , n=1\ ,\\
1<p \leq 2\ , n=2\ ,\\
p=1, \ n=3\ ,
\end{cases}
\end{split}
\eneq
and 
\beeq
\nonumber
\begin{split}
1<pq\leq \frac{(n+1)p}{(n-1)p-2}\ ,
\begin{cases}
2<p\leq 3\ , n=2\ ,\\
1<p\leq 2\ , n=3\ ,\\
1<p\leq \frac{n+1}{n-1}\ , n\geq 4\ .
\end{cases}
\end{split}
\eneq
Recently, Ikeda-Sobojima-Wakasa \cite{I-S-W} obtained the upper bound estimates of lifespan on a larger range (except $n=1$) 
\beeq
\label{Ga}
\Gamma_{GG}(p, q, n):=\max\{ \frac{p+1}{pq-1}, \frac{q+1}{pq-1}\}-\frac{n-1}{2}\geq 0\ ,\  p, q>1\ ,n\geq 2\ .\eneq
For global existence, Kubo-Kubota-Sunagawa \cite{KKS} obtained the radially symmetric global solution in three spatial dimensions under the conditions 
\beeq\label{817}
\frac{q+1}{pq-1}<1\ , 1<p\leq q\ , n=3\ .
\eneq
It can be observed that the condition $\Gamma_{GG}(p, q, n)\geq 0$ is sharp, at least for $n=3$. 

\begin{figure}%[h]
\centering
\begin{tikzpicture} 
%\draw[thick] (2,2)--(3,1) (3,2)--(5,2);

\filldraw[black!10!] (1,4)--(1,1)--(2,2)

to[out=30,in=260](1.9,2.22 )--(1.85,2.353)--(1.8, 2.5)--(1.7,2.857)--(1.6, 3.33)--(1.5,4);

\filldraw[black!30!] (1,4)--(1,1)--(1.6225,1.6225)

to[out=30,in=260](1.6,1.667)--(1.5,2)--(1.4,2.5)--(1.3,3.33)--(1.27,3.7)--(1.25,4);

\node[below] at (1.3,1.88){\tiny $c\neq 1$};

\node[below] at (1.4,3){\tiny $c= 1$};

\draw[thick, domain=1.5:3] plot (\x, {{2/((\x)-1)}});

\draw[thick, domain=1:3.5] plot (\x, {\x});

\node[below] at (3.5,3) {$p=q$};

\draw[thick, dashed, domain=1.25:2] plot (\x, {{1/((\x)-1)}});

%\draw[fill=black,line width=1pt] (2.414,2.414) circle[radius=0.3mm];
%\draw[fill=black,line width=1pt] (3,2) circle[radius=0.3mm];
\node[below left] at (1,1) {\tiny $(1,1)$};
\node[left] at (1,2) {$2$};
%\node[right] at (2,1) {$Z$};
%\node[left] at (1,1) 
%\node[below] at (2,1) {\small  $p_G$};%{$1+\sqrt{2}$};
\node[below] at (2,1) {\small $2$};
\node[below] at (3,1) {\tiny $3$};
	%\draw[thick] (1,1)--(2.414,2.414) (1,1)--(3,2);
	\draw[thick,-stealth] (1,1)--(1,4) node[left]{$q$};
	\draw[thick,-stealth] (1,1)--(5,1) node[below]{$p$};
\node[below] at (3,0.5) {\tiny $n= 3$, $p\leq q$};

\end{tikzpicture}

\label{tu}
\end{figure}

However, when the wave speed $c\neq 1$, the critical curve of \eqref{1} has the potential to change. Kubo-Kubota-Sunagawa \cite{KKS} obtained the radially symmetric global solution when $\gm=\gm_0$ and 
\beeq\label{817001}
\frac{q}{pq-1}<1\ , 1<p\leq q\ , n=3\ .
\eneq
Compare with \eqref{817} (see figure \ref{tu}), we are convinced that the wave speed does effect the critical curve of \eqref{1}. For the blow up part, there is few results. Xu \cite{X} obtained the blow up results under the condition
\beeq\label{81701}
\frac{p+q+2}{pq-1}\geq 2n\ , 
\eneq
when $\gm=\gm_0$. We can see that there is a gap between \eqref{81701} and  \eqref{817001} when $n=3$. 
 In this work, we fill the gap and find that the weak solution of system \eqref{1} under the condition 
\beeq
\Ga_{GG}^{*}(p, q, n)=\max\{\frac{p}{pq-1}\ , \frac{q}{pq-1}\}-\frac{n-1}{2}> 0\ , n\geq 2\ ,
\eneq
will blow up at finite time whenever $c\neq 1$. 
%%It is obvious that 

\begin{thm}
\label{thm2}
Let $n\geq 2$, $0<c\in \R$ and  
\beeq
\label{2re}
\begin{cases}
\Ga_{GG}^{*}(p, q, n)>0 \ , c\neq1\ ,\\
  \Ga_{GG}(p, q, n)>0 \ , c=1\ .
\end{cases}
\eneq 
Consider the system \eqref{1} posed on asymptotically Euclidean manifolds \eqref{adl1}-\eqref{unelp}.
 Assuming that
the data are nontrivial and satisfying 
 \begin{align}
 \label{807}
 \begin{split}
 \begin{cases}
 \int_{\R^n} (u_1-c\la u_0)\phi_{\la}(x)d v_{\gm}>0\ ,\\
 \int_{\R^n} (v_1-\la v_0)\phi_{\la}(x)d v_{\gm}>0\ .
\end{cases}
\end{split}
 \end{align}
Suppose it has a weak solution $u$, $v\in C^{2}([0, T_\ep);\mathcal{D}'(\R^{n}))$ with  $|v_t|^{p}$, $|u_t|^{q}$$\in$ $ C([0, T_\ep);\mathcal{D}'(\R^{n}))$ and satisfying \eqref{fspyl}. Then 
there exists a constant $\ep_0>0$, such that for any $\ep\in (0, \ep_{0})$, the solutions will blow up at finite time. Moreover, we have the estimate for the lifespan $T_\ep$
\beeq
\begin{split}
T_{\ep}\ \les
\begin{cases} 
   \ep^{-\Ga_{GG}(q, p ,n)^{-1}} \ , c=1\ ,\\
\ep^{-\Ga_{GG}^{*}(p, q ,n)^{-1}}\ , c\neq 1\ .
\end{cases}
\end{split}
\eneq
\end{thm}
\begin{rem}
Based on the work of Kubo-Kubota-Sunagawa \cite{KKS}, we conjecture that the critical curve to system \eqref{2} is $\Ga_{GG}^{*}(p, q, n)=0$ whenever $c\neq 1$.
\end{rem}

The last kind of system that we studied is related to both the Strauss conjecture and the Glassey conjecture. More precisely, we consider the following system 
 \beeq
\label{3}
\begin{cases}
\pa^2_t u-c^2\Delta_{\gm} u=|v|^p \ , (t, x)\in (0, T)\times \R^n\ ,\\
\pa^2_t v-\Delta_{\gm} v=|u_t|^q\ , (t, x)\in (0, T)\times \R^n\ ,\\
u(0, x)=\ep u_0(x), \ u_t(0, x)=\ep u_1(x)\ ,\\
v(0, x)=\ep v_0(x), \ v_t(0, x)=\ep v_1(x)\ .\\
\end{cases}
\eneq

When $c=1$ and $\gm=\gm_0$, Hidano-Yokoyama \cite{HY} obtained the finite time blow up result when 
$$\frac{p+1+q^{-1}}{pq-1}-\frac{n-1}{2}>0\ , q<2n/(n-1)\ , n\geq 2\ .$$
Later, Ikeda-Sobojima-Wakasa \cite{I-S-W} obtained the finite time blow up as well as the upper bound estimates of lifespan on a larger range
\beeq
\Ga_{SG}(p,q,n):=\max\{\frac{p+1+q^{-1}}{pq-1}\ , \frac{2+p^{-1}}{pq-1}\}-\frac{n-1}{2}\geq 0\ , n\geq 2\ .
\eneq
Dai-Fang-Wang \cite{DFW} showed that there are no global weak solutions of \eqref{3} when 
\beeq
\frac{q+1}{pq-1}-\frac{n-1}{2}>0\ , \frac{(n-1)(p-1)}{2}<1\ , n\geq 2\ ,
\eneq
and the global existence in three dimension with 
$$\frac{q+1+1/(pq)}{pq-1}<1\ , 1<p<2, \ 2<q<3\ , n=3\ .$$
To the best of the author's knowledge, there are no results for multiple speeds cases. 
In this paper, we obtained the blow up results of system \eqref{3} with 
\beeq\nonumber
\begin{cases}
M(p, q, n):=\min \{\Ga_{GG}(p, q, n)\ , \Ga_{SG}(p,q,n)\}>0\ ,\  c=1\ , \\
M^{*}(p, q, n):=\min\{\Ga^{*}_{GG}(p, q, n),\Ga_{SG}(p,q,n)\}>0 \ ,\  c\neq1\ .
\end{cases}
\eneq

\begin{thm}
\label{thm3}
Let $n\geq 2$, $0<c\in \R$ and 
\beeq\label{hunhe}
\begin{cases}
M(p, q, n)>0\ , \ c=1\ , \\
M^*(p, q, n)>0\ ,\  c\neq1\ .
\end{cases}
\eneq
Consider the system \eqref{3} posed on asymptotically Euclidean manifolds \eqref{adl1}-\eqref{unelp}.
 Assuming that
the data are nontrivial satisfying 
 \begin{align}
 \label{732}
 \begin{split}
 \begin{cases}
 \int u_i dv_{\gm} >0\ , \int_{\R^n} u_0\phi_{\la}(x)d v_{\gm}>0\ , \int_{\R^n} (u_1-c\la u_0)\phi_{\la}(x)d v_{\gm}>0\ , \\
\int v_i dv_{\gm} >0\ ,\int_{\R^n} v_0\phi_{\la}(x)d v_{\gm}>0\ , \int_{\R^n} (v_1-\la v_0)\phi_{\la}(x)d v_{\gm}>0\ ,
\end{cases}
\end{split}
 \end{align}
 where $i=0, 1$. Suppose it has a weak solution $u, v\in C^{2}([0, T_\ep);\mathcal{D}'(\R^{n}))$ with  $|v|^{p}, |u_t|^{q}\in C([0, T_\ep);\mathcal{D}'(\R^{n}))$ and satisfying \eqref{fspyl}. Then 
there exists a constant $\ep_0>0$, such that for any $\ep\in (0, \ep_{0})$, the solutions will blow up at finite time. Moreover, we have the estimate for the lifespan $T_\ep$
\begin{align*}
T_{\ep}\ \les\  
\begin{cases}
\ep^{-M(q, p ,n)^{-1}} \ , M(p, q, n)>0\ ,\ c =1\ ,\\
\ep^{-M^*(q, p ,n)^{-1}} \ ,M^*(p, q, n)>0\ ,\  c \neq 1\ .
\end{cases}
\end{align*}

\end{thm}

 \subsubsection*{Outline} 
 
Our paper is organized as follows. In Section \ref{sec:1stTest},
we list the existence results of special solutions for the elliptic ``eigenvalue" problems \eqref{1.1}, with certain asymptotic behavior.
These solutions play important role in constructing the test functions and the proof of blow up results. Moreover, we list a Lemma in \cite{SGM} for ODE system which can be viewed as a vector version of Kato type Lemma.  In Sections \ref{proof}-\ref{sec5}, we give the proof of Theorems \ref{thm4}-\ref{thm3} by applying iteration method (see, e.g., \cite{LT}).

  \section{Preliminary}\label{sec:1stTest} 
In this section, we list some Lemmas we shall use later.  We first construct
special positive standing wave solutions, of the form $$\psi_1(t, x)=e^{-c\la t}\phi_{\la}(x), \ \psi_2(t, x)=e^{-\la t}\phi_{\la}(x)\ ,$$
 to the linear problem,
$$\pa^{2}_t u-c^2\Delta_{\gm} u=0\ , \pa^{2}_t v-\Delta_{\gm} v=0\ ,$$
with the desired asymptotic behavior. In turn, it is reduced to constructing solutions to certain elliptic ``eigenvalue" problems:
 \beeq
\label{1.1}
\Delta_{\gm}\phi_\la=\lambda^{2}\phi_\la\ .
\eneq
These solutions  will play a key role in the construction of the test functions. 
  \begin{lem}[Lemma 3.1 in \cite{LW2019} ]\label{elp}
Let $n\geq 2$ and 
$(\R^n, \gm)$ be asymptotically Euclidean manifold with \eqref{adl1}-\eqref{unelp}.
Then there exist $\la_0, c_0>0$ such that for any
$0<\la\le \la_0$, there is a solution of \eqref{1.1} satisfying
\beeq
\label{1.50}
c_{0} %\< \la r\>^{-\frac{n-1}{2}}e^{\la \int^{r}_{0}K(\tau)d\tau}
<\phi_\la(x) < c_0^{-1}\< \la r\>^{-\frac{n-1}{2}}e^{\la \int^{r}_{0}K(\tau)d\tau} \ .
\eneq
\end{lem}

With the asymptotic behavior of $\phi_\la$ in hand, we could have the following estimates for $\psi_1, \psi_2$.

%\beeq
%D\subset\{(t,x); |x|\leq \frac{t}{\delta_{0}\delta_{1}}+\frac{R_{1}}{\delta_{0}}\}=\tilde{D}\ .
%\eneq

\begin{lem}
\label{le1}
Let $n\geq 2$, $p, q>1$. Then we have
\begin{align*}
&\int_{D_i}\psi_{i}^{p}d v_{\gm} \ \les\  (t+1)^{n-1-\frac{n-1}{2}p}\ , i=1, 2\ ,\\
%\int_{D_2} \psi_1^{p'}dv_{\gm}\ \les\  e^{\la(1-c)p't}(t+1)^{n-1-\frac{n-1}{2}p'}\\
&\int_{D_2}\psi_{1}^{-\frac{p'}{p}}\psi_{2}^{p'}d v_{\gm} \ \les 
e^{\frac{\la(c-1)t}{p-1}}(t+1)^{\frac{n-1}{2}}\ ,\\
&\int_{D_1}\psi_{2}^{-\frac{q'}{q}}\psi_{1}^{q'}d v_{\gm} \ \les\  
e^{\frac{\la (1-c)t}{q-1}}(t+1)^{\frac{n-1}{2}}\ .
\end{align*}

\end{lem}
\begin{proof}
We only give a short proof of the second inequality since others follow the same way. 
$$\int_{D_{2}}\psi_{1}^{-\frac{p'}{p}}\psi_{2}^{p'} d v_{\gm}=e^{\frac{\la t(c-p)}{p-1}}\int_{D_{2}} \phi_{\la}(x)d v_{\gm}=e^{\frac{\la (c-1)t}{p-1}}\int_{D_{2}}e^{-\la t}\phi_{\la}d v_{\gm}\ .$$
Then by Lemma 3.2 in \cite{LWJDE} (with $q=1$), we have that 
$$\int_{D_{2}}e^{-\la t}\phi_{\la}d v_{\gm}\  \les \ (t+1)^{\frac{n-1}{2}}\ ,$$
which complete the proof. For the reader's convenience, we present a proof here.
We divide the region $D_2$ into two disjoint parts: $D_2=N_1\cup N_{2}$ where
$$N_{1}=\{(t, x); \int^{|x|}_{0}K(\tau)d\tau\leq \frac{t+R_{1}}{2}\}\ .$$
%$$D_{2}=\{(t, x); \frac{\eta(t)+R_{1}}{2}\leq\int^{|x|}_{0}K(\tau)d\tau\leq \eta(t)+R_{1}\}\ .$$
For the region $N_{1}$, we have 
$$\int_{N_{1}}\phi_\la d v_{\gm}\les e^{-\la t}\int_{N_{1}}(1+\la|x|)^{-\frac{n-1}{2}}e^{\la\int^{|x|}_{0}K(\tau)d\tau}d v_{\gm}\ .$$
Let $\tilde{r}=\int^{|x|}_{0}K(\tau)d\tau$, then $d\tilde{r}=K(r)dr$ and $\delta_{0}r\leq \tilde{r}\leq r/ \delta_{0}$ since $K\in[\delta_{0}, 1/\delta_{0}]$. Then we get 
\begin{align*}
\int_{N_{1}} \phi_\la d v_{\gm}&\les e^{-\la t}\int^{\frac{t+R_{1}}{2}}_{0}(1+\tilde{r})^{n-1-\frac{n-1}{2}}e^{\la\tilde{r}}d\tilde{r}\\
&\les e^{-\la t}\int^{\frac{t+R_{1}}{2}}_{0} e^{\frac{3}{2}\la\tilde{r}}d\tilde{r}\\
&\les e^{-\frac{\la}{4} t} \les (1+t)^{\frac{n-1}{2}},
\end{align*}
where we have used the fact that $e^{-t}$ decays faster than any polynomial. 
For the region $N_{2}$, it is easy to see
\begin{align*}
\int_{N_{2}}\phi_\la d v_{\gm}&\les e^{-\la t}\int^{t+R_{1}}_{\frac{t+R_{1}}{2}}e^{\la\tilde{r}}(1+\tilde{r})^{n-1-\frac{n-1}{2}}d\tilde{r}\\
&\les (1+t)^{\frac{n-1}{2}}\int^{t+R_{1}}_{\frac{t+R_{1}}{2}}e^{\la\tilde{r}-\la t}d\tilde{r}\\
&\les (1+t)^{\frac{n-1}{2}}\ .
\end{align*}
\end{proof}

At the end of this section, we list a Lemma for ODE systems to blow up in finite time. This Lemma can be viewed as a vector version of the Kato type Lemma. 
\begin{lem}[Lemma 2.1 in \cite{SGM}]
\label{le2}
Let $a, b\in [0, \infty]$ with $a<b$ and $M(t), N(t)\in C([a, b]; \R)$. Suppose that for all $t\in [a, b]$, we have 
\beeq
\begin{split}
\begin{cases}
M(t)\geq C(t+R)\ ,\\
N(t)\geq C(t+R)^{s}\ , \\
M''(t)\geq C(t+R)^{-\al}(N(t))^e\ ,\\
N''(t)\geq C(t+R)^{-\be}(M(t))^l\ ,
\end{cases}
\end{split}
\eneq
where $C, R>0$, $s\geq 1$, $\al, \be>0$, $e, l>0$. Moreover, if $el>1$ and 
$$l(\al-2)+\be-2< s(el-1)\ ,$$
then $b<\infty$. 

\end{lem}

%\section{Test functions}\label{sec:AE3}
%In this section, we construct the test function and collect the properties of two auxiliary functions, which we shall use later.
%\subsection{Solution of linear dual equation}

 \section{Proof of Theorems \ref{thm4}}\label{proof}
 In this section, we give the proof the Theorem \ref{thm4}. Without loss of generality, we assume $0<c\leq1$. In this case, we see that $D_1\subset D_2$. We first define functionals
  $$F(t)=\int_{D_2} u dv_{\gm}\ , \ G(t)=\int_{D_2}v dv_{\gm}\ .$$
Then by the system \eqref{2}, we know that
$$F''(t)\footnote{Note that $u|_{\pa D_1}=0$, hence $u|_{\pa D_2}=0$.}=\int_{D_2} |v|^{p}dv_{\gm}\ , G''(t)=\int_{D_2} |u|^{q}dv_{\gm}\ .$$
To connect $F''$, $G''$ with $F$, $G$, we apply H\"older's inequality to get 
$$F''=\int_{D_2} |v|^{p}dv_{\gm} \gtrsim \frac{\big(\int_{D_2}v dv_{\gm}\big)^p}{(t+1)^{n(p-1)}}\ , \ G''=\int_{D_2} |u|^{q}dv_{\gm} \gtrsim \frac{\big(\int_{D_2}u dv_{\gm}\big)^q}{(t+1)^{n(q-1)}}$$ 
Hence we can get the following ODE system 
\beeq
\label{die8}
F''\geq c_0\frac{G^p}{(t+1)^{n(p-1)}}\ , 
G''\geq c_1 \frac{F^q}{(t+1)^{n(q-1)}}\ ,
\eneq
for some constants $c_0, c_1>0$. By the assumption \eqref{73} on the initial data, we have the lower bound of $F$, $G$
\beeq
F(t)\gt \ep(t+1)\ ,  \ G\gt \ep(t+1), \ t>0\ .
\eneq
Define auxiliary functions
$$H_1(t)=\int_{D_1} u\psi_1 dv_{\gm}\ , \ H_2(t)=\int_{D_2} v\psi_2 dv_{\gm}\ .$$We will see that 
the lower bound of $F$, $G$ could be improved to
\beeq
\label{die7}
F(t) \gtrsim   \ep^p (t+1)^{n+1-\frac{n-1}{2}p}\ , G(t)\gtrsim  \ep^q (t+1)^{n+1-\frac{n-1}{2}q}\ , \ t>0\ ,
\eneq
under the condition $$n-\frac{n-1}{2}p>0\ , n-\frac{n-1}{2}q>0\ .$$
 In fact, by \eqref{2}, we have 
$$H''_2+2\la H'_{2}=\int_{D_2} |u|^{q}\psi_2 dv_{\gm}\geq 0\ ,$$
thus by the assumption \eqref{73} on the initial data , 
$$H_2'(t)\geq e^{-2\la t}H_2'(0)\geq 0\ , H_2(t)\geq H_2(0) \gtrsim \ep\ .$$
By applying H\"older's inequality and Lemma \ref{le1}, we get
\begin{align*}
\ep\les H_2\les &\Big(\int_{D_2} |v|^{p}dv_{\gm}\Big)^{1/p}\Big(\int_{D_2} |\psi_2|^{p'}dv_{\gm}\Big)^{1/p'}\\
\les &\Big(\int_{D_2} |v|^{p}dv_{\gm}\Big)^{1/p}\Big((t+1)^{n-1-\frac{n-1}{2}p'}\Big)^{1/p'}\ .
\end{align*}
Then
$$F''=\int_{D_2} |v|^{p} dv_{\gm}\gtrsim \ep^p (t+1)^{n-1-\frac{n-1}{2}p}\ .$$
Similarly, we see that 
$$H''_1+2\la cH'_{1}=\int_{D_1} |v|^{q}\psi_1 dv_{\gm}\geq 0\ ,$$
thus 
$$H_1'(t)\geq e^{-2\la t}H_1'(0)\geq 0\ , H_1(t)\geq H_1(0) \gtrsim \ep\ .$$
By applying H\"older's inequality and Lemma \ref{le1}, we get
\begin{align*}
\ep\les H_1\les &\Big(\int_{D_1} |u|^{q}dv_{\gm}\Big)^{1/q}\Big(\int_{D_1} |\psi_1|^{q'}dv_{\gm}\Big)^{1/p'}\\
\les &\Big(\int_{D_1} |u|^{q}dv_{\gm}\Big)^{1/q}\Big((t+1)^{n-1-\frac{n-1}{2}q'}\Big)^{1/q'}\ ,
\end{align*}
which yields
$$ G''=\int_{D_2}|u|^{q} dv_{\gm}\geq \int_{D_1}|u|^{q} dv_{\gm}\gtrsim \ep^q (t+1)^{n-1-\frac{n-1}{2}q}\ .$$
Hence, we could integrate the 
above inequalities twice to get \eqref{die7}.

Applying Lemma \ref{le2} to \eqref{die8} and \eqref{die7}, we know that the solutions will blow up if $\Ga_{SS}(p, q, n)>0$. In fact, by the second inequality of \eqref{die8}, we have 
$$G(t) \gtrsim (t+1)^{\big(n+1-\frac{n-1}{2}p\big)q-n(q-1)+2}\ .$$
Heuristically, 
the lower bound of $F$ could be improved by the first inequality of \eqref{die8} if 
$$\Big(\big(n+1-\frac{n-1}{2}p\big)q-n(q-1)+2\Big)p-n(p-1)+2> n+1-\frac{n-1}{2}p\ ,$$
similarly, the lower bound of $G$ could be improved if $$\Big(\big(n+1-\frac{n-1}{2}q\big)p-n(p-1)+2\Big)q-n(q-1)+2> n+1-\frac{n-1}{2}q\ ,$$
that is $\Ga_{SS}(p, q, n)>0$. Without loss of generality, we assume $p\leq q$ and in this case 
\beeq
\Ga_{SS}(p, q, n)=\frac{p^{-1}+q+2}{pq-1}-\frac{n-1}{2}>0\ .
\eneq
To get the upper bound lifespan estimates, we apply iteration method and we will see that there exists $\ep_0>0$ such that for any $\ep\in (0, \ep_0)$, we have the lower bound of $F$, $G$ by some sequences
$$F\geq A_j(t+1)^{a_j}\ , G\geq B_j(t+1)^{b_j}\ , j\in\N\ ,$$
with initial lower bound
 $$A_1=\ep^p\ , a_1=n+1-\frac{n-1}{2}p\ .$$
In fact, by putting the lower bound of $F$ to the second inequality of \eqref{die8}, we have 
$$G''\geq c_1A^{q}_1(t+1)^{qa_1-n(q-1)}=c_1\ep^{pq}(t+1)^{qa_1-n(q-1)}\ .$$ 
%Without loss of generality, we assume $qa_1-n(q-1)+1>0$. 
Note that $G(0), G'(0)\sim \ep$, then we have
\begin{align*}
G'\geq &\frac{c_1\ep^{pq}}{qa_1-n(q-1)+1}(t+1)^{qa_1-n(q-1)+1}-\frac{c_1\ep^{pq}}{qa_1-n(q-1)+1}+G'(0)\\
\geq &\frac{c_1\ep^{pq}}{qa_1-n(q-1)+1}(t+1)^{qa_1-n(q-1)+1}\ ,
\end{align*}
and 
\begin{align*}
G\geq \frac{c_1\ep^{pq}(t+1)^{qa_1-n(q-1)+2}}{\big(qa_1-n(q-1)+1\big)\big(qa_1-n(q-1)+2\big)}
\end{align*}
if we take $\ep$ small enough. Hence we obtain that
$$B_1=\frac{c_1\ep^{pq}}{b_1\big(b_1-1\big)}\ , b_1=qa_1-n(q-1)+2\ ,$$
and then we use above to get $A_2$, $a_2$ by the first inequality of \eqref{die8}. Following this iteration way, we can obtain the following iteration equation
\begin{align}
\label{byzd}
\begin{cases}
A_j =\frac{c_0B^{p}_{j-1}}{(a_j-1)a_j}\ , a_j =pb_{j-1}-n(p-1)+2\ , j\geq 2\\
\\
 B_{j}=\frac{c_1A^{q}_j}{(b_j-1)b_j}\ , b_{j}=q a_j-n(q-1)+2\ , j\geq 1
\end{cases}
\end{align}
if we have 
\beeq
\label{8052}
(a_j+1)A_j \leq \ep\ ,\ (b_j +1)B_j\leq \ep\ , \forall j\in \N\ .
\eneq
In fact, by \eqref{byzd} we get that
$$
b_j=\Big(\frac{p^{-1}+q+2}{pq-1}-\frac{n-1}{2}\Big)(pq)^{j}+n-\frac{2q+2}{pq-1}\ .$$
Since $b_j$ and $a_j$ are increasing to $\infty$ as $j\to\infty$, there exists $N_0>0$ such that $a_j, b_j\geq 2$ for any $j\geq N_0$. Thus we have 
\begin{align}
\label{8051}
\begin{cases}
(b_j +1)B_j=\frac{c_1A^{q}_j}{(b_j-1)b_j}+\frac{c_1A^{q}_j}{(b_j-1)}\leq 2c_1A^{q}_j\leq 2c_1c^{q}_0B^{pq}_{j-1}\ , j\geq N_0+1\ , \\
(a_j+1)A_j=\frac{c_0B^{p}_{j-1}}{(a_j-1)a_j}+\frac{c_0B^{p}_{j-1}}{(a_j-1)}\leq 2c_0B^{p}_{j-1}\leq 2c_0c^{p}_1A^{pq}_{j-1}\ , j\geq N_0+1\ .
\end{cases}
\end{align}
For $j\leq N_0$, we take $\ep_0$ small enough such that for any $\ep\in(0, \ep_0)$, we have 
\beeq
\label{805}
(b_j +1)B_j+(a_j +1)A_j \leq \frac{\ep}{2c_0c_1^{p}+2c_1c_0^{q}+1}\ .
\eneq
Combine \eqref{805} and \eqref{8051}, we get \eqref{8052}. 

For the coefficient $B_j$, it is easy to see that
$$
B_j=\frac{c_1c^{q}_0B_{j-1}^{pq}}{\big((a_j-1)a_j\big)^{q}(b_j-1)b_j}\geq\frac{c_1c^{q}_0B_{j-1}^{pq}}{\big(p b_j+2\big)^{2q+2}}\geq \frac{MB_{j-1}^{pq}}{(pq)^{j(2q+2)}}\ ,
$$
where
$$M=c_1c^{q}_0\Big(2+\frac{n+1}{2}p+\frac{1+pq+2p}{pq-1}\Big)^{-(2q+2)}\ .$$
Hence we get 
\begin{align*}
\ln B_j &\geq (pq)\ln B_{j-1} +\ln M-j(2q+2)\ln (pq)\\
&\geq (pq)^2\ln B_{j-2}+(pq+1)\ln M-\big(pq(j-1)+j\big)(2q+2)\ln (pq)\\
&\geq (pq)^{j-1}\ln B_1-(pq)^{j-1}\sum^{j-1}_{k=1}\frac{(k+1)(2q+2)\ln (pq)-\ln M}{(pq)^{k}}\ .
\end{align*}
Note that the series is convergence, that is
$$S(\infty)=\sum^{\infty}_{k=1}\frac{(k+1)(2q+2)\ln (pq)+|\ln M|}{(pq)^{k}}<\infty\ .$$
Then we have 
$$\ln B_j\geq (pq)^{j-1}(\ln B_1-S(\infty))\ .$$
Thus 
\begin{align*}
G&\geq B_{j}(t+1)^{b_j}\\
&\geq e^{(pq)^{j-1}\Big(\ln B_1-S(\infty)+\big(\frac{2q+2}{pq-1}+q+2-\frac{n-1}{2}pq\big)\ln(t+1)\Big)}e^{(n-\frac{2q+2}{pq-1})\ln (t+1)}\ , \forall j\in \N\ ,
\end{align*}
which yields the lifespan estimate
$$T_{\ep} \les\ \ep^{-\big(\frac{q+2+p^{-1}}{pq-1}-\frac{n-1}{2}\big)^{-1}}\ \ .$$ 
Similarly, if we assume $p\geq q$ that is 
$$\Ga_{SS}(p, q, n)=\frac{q^{-1}+p+2}{pq-1}-\frac{n-1}{2}>0\ ,$$
and start the iteration with $B_1=\ep^q\ , b_1=n+1-\frac{n-1}{2}q$, we will have that 
$$T_\ep \les\ \ep^{-\big(\frac{p+2+q^{-1}}{pq-1}-\frac{n-1}{2}\big)^{-1}}\ .$$
Hence, we obtain the lifespan estimate
$$T_\ep \ \les \ \ep^{-\Ga_{SS}(p,q,n)^{-1}}\ .$$

\section{Proof of Theorem \ref{thm2}}

Without loss of generality, we assume $0<c\leq1$. In this case, we see that $D_1\subset D_2$. We first define functionals 
\beeq
\label{dl1}
F(t)=\int_{D_2} \big(u_t\psi_1+c\la u\psi_1\big)dv_{\gm}\ , G(t)=\int_{D_2} \big(v_t\psi_2+\la v\psi_2\big)dv_{\gm}\ .
\eneq
Then by the system \eqref{1}, we have 
\beeq\label{813}
F'(t)=\int_{D_2} |v_t|^{p}\psi_1 dv_{\gm}\ , \ G'(t)=\int_{D_2} |u_t|^{q}\psi_2 dv_{\gm}\ .
\eneq
In order to relate $F'$ with $G$, we let $H_1(t)=\int_{D_2} u_t \psi_1 d v_{\gm}$ and it is easy to check that 
$$(2H_1-F)'+2c\la (2H_1-F)=\int_{D_2} |v_t|^{p}\psi_1 d v_{\gm}\geq 0\ , $$
thus by the assumption on the initial data \eqref{807}, we have 
$$(2H_1-F)(t)\geq (2H_1-F)(0)\geq 0\ ,$$ 
which yields 
$$\frac{F}{2}\leq H_1(t)=\int_{D_2} u_t\psi_1dv_{\gm}\ .$$
Similarly, let $H_2=\int_{D_2} v_t\psi_2 dv_{\gm}$, we can get 
$$ \frac{G}{2}\leq H_2(t)= \int_{D_2} v_t\psi_2 dv_{\gm} \ .$$
Recall that the support of $u$ is on $D_1$ by \eqref{fspyl}, so $H_1(t)=\int_{D_1} u_t \psi_1 d v_{\gm}$. 
By applying H\"older's inequality to $H_1$ and $H_2$, we have
\begin{align*}
\int_{D_{2}} |v_t|^{p}\psi_1dv_{\gm} \gtrsim \frac{\Big(\int_{D_2} v_t\psi_2 dv_{\gm}\Big)^{p}}{\Big(\int_{D_2}  \psi_{1}^{-\frac{p'}{p}}\psi^{p'}_2 dv_{\gm}\Big)^{p/p'}} \gtrsim\frac{H_{2}^p}{\Big(\int_{D_2}  \psi_{1}^{-\frac{p'}{p}}\psi^{p'}_2 dv_{\gm}\Big)^{p/p'}}\ , 
\end{align*}
\begin{align*}
\int_{D_{1}} |u_t|^{q}\psi_2dv_{\gm} \gtrsim \frac{\Big(\int_{D_1} u_t\psi_1 dv_{\gm}\Big)^{q}}{\Big(\int_{D_1}  \psi_{2}^{-\frac{q'}{q}}\psi^{q'}_1 dv_{\gm}\Big)^{q/q'}}\gtrsim 
\frac{H_{1}^q}{\Big(\int_{D_1}  \psi_{2}^{-\frac{q'}{q}}\psi^{q'}_1 dv_{\gm}\Big)^{q/q'}} \ .
\end{align*}
 \subsection{$c=1$}
When $c=1$, we have $D=D_1=D_2$, by Lemma \ref{le1}, 
$$\Big(\int_{D_2}  \psi_{1}^{-\frac{p'}{p}}\psi^{p'}_2 dv_{\gm}\Big)^{p/p'}\gtrsim (t+1)^{(n-1)(p-1)/2}\ ,$$
$$\Big(\int_{D_2}  \psi_{2}^{-\frac{q'}{q}}\psi^{q'}_1 dv_{\gm}\Big)^{q/q'} \gtrsim (t+1)^{(n-1)(q-1)/2}\ .$$
Thus we get the following ODE system 
\beeq
\label{die4}
\begin{cases}
F'(t) \geq \frac{c_0|G|^p}{(t+1)^{\frac{(n-1)(p-1)}{2}}}\ , F(0)\geq c_0\ep\  \\
\\
G'(t) \geq \frac{c_1|F|^q}{(t+1)^{\frac{(n-1)(q-1)}{2}}}\ , G(0)\geq c_1\ep \ . \\
\end{cases}
\eneq
where constants $c_0$ and $c_1$ depends on $p, q, c, \la, n$. We will see that \eqref{die4} blows up in $\Ga(p, q, n)>0$. In fact, it is easy to see $F(t)\geq F(0)\geq c_0\ep$ due to $F' \geq 0$. By putting it into the second inequality in \eqref{die4}, we have
\beeq
\label{die5}
G'(t)\geq c_1(c_0\ep)^{q}(t+1)^{-\frac{(n-1)(q-1)}{2}} \ ,
\eneq
thus $G(t) \gtrsim (t+1)^{1-(n-1)(q-1)/2}$ and by \eqref{die4} again, we get that
$$F'(t) \gtrsim (t+1)^{\big(1-\frac{(n-1)(q-1)}{2}\big)p-\frac{(n-1)(p-1)}{2}}\ .$$
Heuristically, the lower bound of $F(t)$ could be improved if 
$$
\Big(1-\frac{(n-1)(q-1)}{2}\Big)p-\frac{(n-1)(p-1)}{2}+1>0\ ,
$$
similarly, the lower bound of $G(t)$ could be improved if 
$$
\Big(1-\frac{(n-1)(p-1)}{2}\Big)q-\frac{(n-1)(q-1)}{2}+1>0\ ,
$$
that is the range $\Ga_{GG}(p, q, n)>0$, which suggests the blow up result.

To obtain the lifespan estimates, we explore iteration argument and assume
\beeq
\label{dl2}
F(t)\geq A_j (t+1)^{a_j}\ , G(t)\geq B_j (t+1)^{b_j}\ , j\in \N\ .
\eneq
Without loss of generality, we assume $\frac{(n-1)(q-1)}{2}<1$, then by \eqref{die4}, we get that 
\begin{align*}
G(t)\geq &\frac{c_0(c_1\ep)^{q}}{1-\frac{(n-1)(q-1)}{2}}(t+1)^{1-\frac{(n-1)(q-1)}{2}} -\frac{c_0(c_1\ep)^{q}}{1-\frac{(n-1)(q-1)}{2}}+c_1 \ep \\
\geq&\frac{c_0(c_1\ep)^{q}}{1-\frac{(n-1)(q-1)}{2}}(t+1)^{1-\frac{(n-1)(q-1)}{2}}\ , \ t> 0\ ,\end{align*}
if we take $\ep$ small enough. We start the iteration with 
$$B_1=\frac{c_0(c_1\ep)^{q}}{1-\frac{(n-1)(q-1)}{2}}\ , \ b_1=1-\frac{(n-1)(q-1)}{2}\ ,$$
and assume 
$$\Ga_{GG}(p, q, n)=\frac{p+1}{pq-1}-\frac{n-1}{2}>0\ .$$
By putting it to \eqref{die4}, we have 
$$F'(t)\geq c_0 B^{p}_1(t+1)^{pb_1-\frac{(n-1)(p-1)}{2}}\ ,$$
that is 
$$F(t)\geq \frac{c_0 B^{p}_1}{1+pb_1-\frac{(n-1)(p-1)}{2}}(t+1)^{1+pb_1-\frac{(n-1)(p-1)}{2}}\ , t>0\ ,$$
if we take $\ep$ small enough. Then we have 
$$A_1=\frac{c_0 B^{p}_1}{1+pb_1-\frac{(n-1)(p-1)}{2}}\ ,\  a_1=1+pb_1-\frac{(n-1)(p-1)}{2}\ .$$
Following this way, we have the iteration equation 
\beeq
\label{s1}
\begin{cases}
A_j=\frac{B^{p}_{j}}{1+pb_j-\frac{(n-1)(p-1)}{2}}\ , \ a_j=1+pb_j-\frac{(n-1)(p-1)}{2}\ , \ j\geq 1\ ,
\\
\\
B_{j+1}=\frac{A^{q}_j}{1+qa_j-\frac{(n-1)(q-1)}{2}}\ ,\ b_{j+1}=1+qa_j-\frac{(n-1)(q-1)}{2}\ , \ j\geq 1\ ,
\end{cases}
\eneq
If we have 
\beeq
\label{yizhi1}
A_j\leq c_0\ep\ , B_j\leq c_1\ep\ , \forall \ j\in \N\ .
\eneq
In fact, by \eqref{s1} we have
$$a_j=\Big(\frac{p+1}{pq-1}-\frac{n-1}{2}\Big)\big((pq)^{j}-1\big)\ , j\in \N\ ,$$
and $a_j \to \infty$ as $j\to \infty$. Thus there exist a $N_0>0$, such that for any $j\geq N_0 +1$, 
$$
A_j =\frac{A^{pq}_{j-1}}{a_j b_{j}^p}\leq A^{pq}_{j-1}\ ,  B_j =\frac{B^{pq}_{j-1}}{a_{j-1}^q b_{j}}\leq B^{pq}_{j-1}\ . $$
For $j\leq N_0$, we take $\ep_0\ll1$ small enough to get 
$$A_j\leq c_0\ep\ , \ B_j\leq c_1\ep\ , j\leq N_0\ , $$
for any $\ep\in (0, \ep_0)$. 

For the sequence $A_j$, by definition we have that
$$A_j=\frac{A^{pq}_{j-1}}{a_j b_{j}^p}\geq \frac{A^{pq}_{j-1}}{K(pq)^{j(p+1)}}\ , K=\Big(q\big(\frac{p+1}{pq-1}-\frac{n-1}{2}\big)+1\Big)^{p+1}\ .$$
Hence we get 
\begin{align*}
\ln A_j &\geq (pq)\ln A_{j-1} -\ln K-j(p+1)\ln (pq)\\
&\geq (pq)^2\ln A_{j-2}-(pq+1)\ln K-\big(pq(j-1)+j\big)(p+1)\ln (pq)\\
&\geq (pq)^{j-1}\ln A_1-(pq)^{j-1}\sum^{j-1}_{k=1}\frac{\ln K+(k+1)(p+1)\ln (pq)}{(pq)^{k}}\ .
\end{align*}
Note that the series is convergence, that is
$$S(\infty)=\sum^{\infty}_{k=0}\frac{\ln K+(k+1)(p+1)\ln (pq)}{(pq)^{k-1}}<\infty\ .$$
Then we have 
$$\ln A_j\geq (pq)^{j-1}(\ln A_1-S(\infty))\ .$$
Thus 
\begin{align*}
F&\geq A_{j}(t+1)^{a_j}\\
&\geq e^{(pq)^{j-1}\Big(\ln A_1-S(\infty)+\big(\frac{p+1}{pq-1}-\frac{n-1}{2}\big)pq\ln(t+1)\Big)}e^{-\big(\frac{p+1}{pq-1}-\frac{n-1}{2}\big)\ln (t+1)}\ ,
\end{align*}
we can see that
$$\ln A_1-S(\infty)+\big(\frac{p+1}{pq-1}-\frac{n-1}{2}\big)pq\ln(t+1)\leq 0\ ,$$
otherwise, $F(t)\to \infty$ as $j\to \infty$. From which we have the estimate of $T_\ep$
$$T_{\ep}\ \les \ \ep^{-\big(\frac{p+1}{pq-1}-\frac{n-1}{2}\big)^{-1}} \ .$$
Moreover, when $\Ga_{GG}(p, q, n)=\frac{q+1}{pq-1}-\frac{n-1}{2}>0$, we use the following iteration equation 
\beeq
\nonumber
\begin{cases}
A_{j+1}=\frac{B^{p}_{j}}{a_{j+1}}\ , \ a_{j+1}=1+pb_j-\frac{(n-1)(p-1)}{2}\ , \ j\geq 1\ ,
\\
\\
B_{j}=\frac{A^{q}_j}{b_{j}}\ ,\ b_{j}=1+qa_j-\frac{(n-1)(q-1)}{2}\ , \ j\geq 1\ ,
\end{cases}
\eneq
with
$$A_1=\frac{c^{p}_0c_1\ep}{1-\frac{(n-1)(p-1)}{2}}\ , \ a_1=1-\frac{(n-1)(p-1)}{2}\ .$$
And we can obtain another lifespan estimate
$$T_{\ep}\ \les \ \ep^{-\big(\frac{q+1}{pq-1}-\frac{n-1}{2}\big)^{-1}} \ .$$
In summary, we have 
\beeq
\label{dl3}
T_{\ep}\ \les \min\{ \ep^{-\big(\frac{p+1}{pq-1}-\frac{n-1}{2}\big)^{-1}} \ ,   \ep^{-\big(\frac{q+1}{pq-1}-\frac{n-1}{2}\big)^{-1}}\} \ .
\eneq
\subsection{$0<c< 1$}
By Lemma \ref{le1}, we have 
\beeq
\label{902}
\begin{cases}
\Big(\int_{D_2}  \psi_{1}^{-\frac{p'}{p}}\psi^{p'}_2 dv_{\gm}\Big)^{p/p'}\ \les \  e^{\la(c-1)t}(t+1)^{(n-1)(p-1)/2}\ ,\\

\Big(\int_{D_1}  \psi_{2}^{-\frac{q'}{q}}\psi^{q'}_1 dv_{\gm}\Big)^{q/q'}\  \les \ e^{\la (1-c)t}(t+1)^{(n-1)(q-1)/2}\ .
\end{cases}
\eneq
Thus we get the following ODE system 
\beeq
\label{die4108}
\begin{cases}
F'(t) \geq \frac{c_0|G|^p}{e^{\la(c-1)t}(t+1)^{\frac{(n-1)(p-1)}{2}}}\ , \ F(0)\geq c_0\ep\  \\
\\
G'(t) \geq \frac{c_1|F|^q}{e^{\la (1-c)t}(t+1)^{\frac{(n-1)(q-1)}{2}}}\ , \ G(0)\geq c_1\ep \  \\
\end{cases}
\eneq
We will see that the system \eqref{die4108} blows up in finite time when $\widetilde{\Ga}(p, q, n)>0$.  For that purpose, we set $G=H(t)e^{\la(c-1)t/p}$, then we have 
\beeq\label{die815}
F'(t) \geq \frac{c_0|H|^p}{(t+1)^{\frac{(n-1)(p-1)}{2}}}\ ,
H'+\frac{\la(c-1)}{p}H \geq \frac{c_1|F|^q}{(t+1)^{\frac{(n-1)(q-1)}{2}}}e^{\la(c-1)t/p'}\ .
\eneq
It is easy to see that $H(t)\geq H(0) =G(0)\geq c_1\ep$. By putting it into the first inequality in \eqref{die815}, we have
$$F(t) \gtrsim (t+1)^{1-(n-1)(p-1)/2}\ ,$$ 
then by \eqref{die815} again, we get that
\begin{align*}&e^{\la(c-1)t/p}H\\ 
\gtrsim &\int^{t}_{t/p}e^{\la (c-1)\tau}(\tau+1)^{\big(1-\frac{(n-1)(p-1)}{2}\big)q-\frac{(n-1)(q-1)}{2}}d\tau\ ,\\
\gtrsim &(t+1)^{\big(1-\frac{(n-1)(p-1)}{2}\big)q-\frac{(n-1)(q-1)}{2}}\int^{t}_{t/p}e^{\la (c-1)\tau}d\tau\\
\gtrsim &(t+1)^{\big(1-\frac{(n-1)(p-1)}{2}\big)q-\frac{(n-1)(q-1)}{2}}e^{\la(c-1)t/p}\Big(1-e^{\la(c-1)t/p'}\Big)\ ,
\end{align*}
Heuristically, the lower bound of $H(t)$ could be improved if 
$$
\Big(1-\frac{(n-1)(p-1)}{2}\Big)q-\frac{(n-1)(q-1)}{2}>0\ ,
$$
that is  
$$\frac{q}{pq-1}>\frac{n-1}{2}\ .$$

While for the case $\frac{p}{pq-1}>\frac{n-1}{2}$, we shall give another intuition. Let $H(t)=e^{\la(c-1)t}F(t)$, then system \eqref{die4108} becomes
\beeq
\label{die4100}
\begin{cases}
H'+\la(1-c)H\geq \frac{c_0|G|^p}{(t+1)^{(n-1)(p-1)/2}}\ , H(0) \gt \ep\  \\
\\
G'(t) \geq \frac{H^q}{(t+1)^{(n-1)(q-1)/2}}\ , \ G(0)\geq c_1\ep \  
\end{cases}
\eneq
By putting $H\gt \ep$ into the second inequality in \eqref{die4100}, we have 
$$G' \gtrsim \ \ep^{q}(t+1)^{-\frac{(n-1)(q-1)}{2}}\ ,$$
thus $$G \gtrsim \ \ep^q (t+1)^{1-\frac{(n-1)(q-1)}{2}}\ .$$
And then by the first inequality in \eqref{die4100}, we get that
$$(e^{\la(1-c)t}H)' \gt \ep^{pq}e^{\la(1-c)t}(t+1)^{\big(1-\frac{(n-1)(q-1)}{2}\big)p-\frac{(n-1)(p-1)}{2}}\ ,$$
thus 
\begin{align*}H(t) \gt &\ep^{pq}(t+1)^{\big(1-\frac{(n-1)(q-1)}{2}\big)p-\frac{(n-1)(p-1)}{2}}e^{\la(c-1)t}\int^{t}_{t/2}e^{\la(1-c)\tau}d\tau\ ,\\
\gt &\ep^{pq}(t+1)^{\big(1-\frac{(n-1)(q-1)}{2}\big)p-\frac{(n-1)(p-1)}{2}}\ ,
\end{align*}
so the lower bound of $H$ could be boosted when 
 $$\big(1-\frac{(n-1)(q-1)}{2}\big)p-\frac{(n-1)(p-1)}{2}>0\ ,$$
 which is actually 
 $$\frac{p}{pq-1}>\frac{n-1}{2}\ .$$

To get the upper bound of lifespan estimates, we explore iteration argument. Without loss of generality, we assume $p\geq q$ and in this case
$$\Ga_{GG}^*(p, q, n)=\frac{p}{pq-1}>0\ .$$
 We assume
$$H(t)\geq A_j (t+1)^{a_j}\ , G(t)\geq B_j (t+1)^{b_j}\ , j\in \N\ .$$
Under this assumption, the first inequality of \eqref{die4100} could be simplified 
as 
\begin{align*}
H(t)\geq &e^{-\la(1-c) t}\int^{t}_{t/2}\frac{c_0e^{\la(1-c) \tau}G^q(\tau)}{(\tau+1)^{(n-1)(p-1)/2}}d\tau+H(0)\\
\geq &\frac{c_1(G(t)/2)^p}{(t+1)^{(n-1)(p-1)/2}}\Big(\frac{1-e^{-\la(1-c) t}}{2\la(1-c)}\Big)
\geq \frac{c_2(G(t)^p}{(t+1)^{(n-1)(p-1)/2}}\ ,
\end{align*}
for some constant $c_2(c_1, \la, p)>0$ and $t\geq 2$. And the system \eqref{die4100} reads
\beeq
\label{die41000}
\begin{cases}
H\geq \frac{c_2 |G|^p}{(t+1)^{(n-1)(p-1)/2}}\ , H(0) \geq c_2\ep\ , \\
\\
G'(t) \geq \frac{c_1H^q}{(t+1)^{(n-1)(q-1)/2}}\ , \ G(0)\geq c_1\ep \ .  
\end{cases}
\eneq
Then by the same way in case $c=1$, there exists $\ep_0>0$ such that for any $\ep\in(0, \ep_0)$, we have the iteration equation 
\beeq
\label{s10}
\begin{cases}
A_{j+1}=c_2B^{p}_{j}\ , \ a_{j+1}=pb_j-\frac{(n-1)(p-1)}{2}\ , j\geq 1\ ,
\\
\\
B_{j}=\frac{c_1A^{q}_j}{b_j}\ ,\ b_{j}=1+qa_j-\frac{(n-1)(q-1)}{2}\ , j\geq 1\ ,
\end{cases}
\eneq
with 
$$A_1=c_2\ep\ , a_1=0\ .$$
Hence we get that 
\begin{align*}
a_j=&\Big(\frac{p}{pq-1}-\frac{n-1}{2}\Big)(pq)^{j-1}-\Big(\frac{p}{pq-1}-\frac{n-1}{2}\Big)\ , j\in \N\ ,\\
A_j=&c_2B^{p}_{j-1}=c^p_{1}c_2\frac{A^{pq}_{j-1}}{b^{p}_{j-1}}\geq K\frac{A^{pq}_{j-1}}{(pq)^{(j-1)p}}\ ,K=c^p_{1}c_2\big(\frac{p}{pq-1}-\frac{n-1}{2}+1\big)^{-p}\ .
\end{align*}
Then we see that 
\begin{align*}
\ln A_j\geq &pq\ln A_{j-1}-(j-1)p\ln(pq)+\ln K\\
&\geq (pq)^2\ln A_{j-2}-p\ln (pq)\big(pq(j-2)+(j-1)\big)+\big(pq+1\big)\ln K\\
&\geq(pq)^{j-1}\ln A_1+(pq)^{j-1}\sum^{j-1}_{m=1}\frac{\ln K-mp\ln (pq)}{(pq)^{m}}\ .
\end{align*}
Note that the series 
$$S(\infty)=\sum^{\infty}_{m=1}\frac{|\ln K|+mp\ln (pq)}{(pq)^{m}}<\infty\ ,$$
thus we get that 
$$\ln A_j\geq(pq)^{j-1}(\ln A_1-S(\infty))\ . $$
Hence we have the lower bound estimate of $H(t)$
\begin{align*}
H(t)&\geq A_j (t+1)^{a_j}=e^{\ln A_j+a_j\ln(t+1)}\\
&\geq e^{(pq)^{j-1}\Big(\ln A_1-S(\infty)+(\frac{p}{pq-1}-\frac{n-1}{2})\ln (t+1)\Big)}e^{-(\frac{p}{pq-1}-\frac{n-1}{2})\ln(t+1)}\ , \forall \  j\in\N\ ,
\end{align*}
which yields the finite time blow up phenomenon if 
$$\Big(\ln A_1-S(\infty)+(\frac{p}{pq-1}-\frac{n-1}{2})\ln(t+1)\Big)>0\ ,$$
so we get the lifespan estimate 
$$T_{\ep} \les\ \ep^{-(\frac{p}{pq-1}-\frac{n-1}{2})^{-1}}\ .$$

\section{Proof of Theorem \ref{thm3}}
\label{sec5}

Without loss of generality, we assume $0<c\leq1$. In this case, we see that $D_1\subset D_2$. 

\subsection{Blow up for $\Ga_{SG}(p, q, n)>0$.}

In this case, we define functionals
 $$F(t)=\int_{D_2} u_t dv_{\gm}\ , \ G(t)=\int_{D_2} v dv_{\gm}\ .$$
Then by the system \eqref{3}, we have
$$F'(t)=\int_{D_2} |v|^{p}dv_{\gm}\ , G''(t)=\int_{D_2} |u_t|^{q}dv_{\gm}\ .$$
To connect $F'$, $G''$ with $F$, $G$, we apply H\"older's inequality to get the following ODE system 
\beeq
\label{die80}
F'\geq c_0\frac{G^p}{(t+1)^{n(p-1)}}\ , 
G''\geq c_1 \frac{F^q}{(t+1)^{n(q-1)}}\ ,
\eneq
for some constants $c_0, c_1>0$. By the assumption \eqref{732} on the initial data, we have the lower bound of $F$, $G$
\beeq
F(t)\gt \ep\ ,  G\gt \ep(t+1), \ t>0\ .
\eneq

By the same procedure in the proof of Theorem \ref{thm4},  
the lower bound of $F$, $G$ could be improved to
\beeq
\label{die70}
F(t) \gtrsim   \ep^p (t+1)^{n-\frac{n-1}{2}p}\ , G(t)\gtrsim  \ep^q (t+1)^{n+1-\frac{n-1}{2}q}\ , t>0\ ,
\eneq
under the condition $n-\frac{n-1}{2}p>0\ , n-\frac{n-1}{2}q>0$, by auxiliary functions
$$H_1(t)=\int_{D_1} u_t\psi_1 dv_{\gm}\ , \ H_2(t)=\int_{D_2} v\psi_2 dv_{\gm}\ .$$
In fact, by \eqref{3}, we have 
$$H''_2+2\la H'_{2}=\int |u_t|^{q}\psi_2 dv_{\gm}\geq 0\ ,$$
thus by the assumption \eqref{732} on the initial data , 
$$H_2'(t)\geq e^{-2\la t}H_2'(0)\geq 0\ , H_2(t)\geq H_2(0) \gtrsim \ep\ .$$
By applying H\"older's inequality and Lemma \ref{le1}, we get
\begin{align*}
\ep\les H_2\les &\Big(\int_{D_2} |v|^{p}dv_{\gm}\Big)^{1/p}\Big(\int_{D_2} |\psi_2|^{p'}dv_{\gm}\Big)^{1/p'}\\
\les &\Big(\int_{D_2} |v|^{p}dv_{\gm}\Big)^{1/p}\Big((t+1)^{n-1-\frac{n-1}{2}p'}\Big)^{1/p'}\ .
\end{align*}
Then
$$F'=\int_{D_2} |v|^{p} dv_{\gm}\gtrsim \ep^p (t+1)^{n-1-\frac{n-1}{2}p}\ ,$$
thus we could integrate the 
above inequality get the lower bound of $F$ in \eqref{die70}. 

For the lower bound of auxiliary function $H_1(t)$, we need another functional $L(t)=\int_{D_1} \big(u_t\psi_1+\la u\psi_1\big)dv_{\gm}$, which appeared in the proof of Theorem \ref{thm2}. It is easy to check that 
$$(2H_1-L)'+2\la c(2H_1-L)=\int_{D_1} |v|^{p}\psi_1 d v_{\gm}\geq 0\ , \ L'=\int |v|^{p}\psi_1 d v_{\gm}\geq 0\ ,$$
since $\supp u\subset D_1$. 
Thus by the assumption on the initial data \eqref{732}, we have 
$$(2H_1-L)(t)\geq (2H_1-L)(0)\gt \ep\ ,\ L(0)\geq 0\ , $$ 
which yields 
$$ H_1(t)\geq \frac{L(t)}{2}+\frac{1}{2}(2H_1-L)(0)\gt \ep\ \ .$$
By applying H\"older's inequality and Lemma \ref{le1}, we get
\begin{align*}
\ep\les H_1\les &\Big(\int_{D_1} |u_t|^{q}dv_{\gm}\Big)^{1/q}\Big(\int_{D_1} |\psi_1|^{q'}dv_{\gm}\Big)^{1/q'}\\
\les &\Big(\int_{D_1} |u_t|^{q}dv_{\gm}\Big)^{1/q}\Big((t+1)^{n-1-\frac{n-1}{2}q'}\Big)^{1/q'}\ ,
\end{align*}
and 
$$ G''=\int_{D_2} |u_t|^{q}dv_{\gm}\geq \int_{D_1} |u_t|^{q}dv_{\gm}\gtrsim \ep^q (t+1)^{n-1-\frac{n-1}{2}q}\ .$$
Hence, we could integrate the 
above inequality twice to get \eqref{die70}. 

Heuristically, 
the lower bound of $F$ in \eqref{die70} could be improved by the ODE system \eqref{die80} if 
$$\Big(\big(n-\frac{n-1}{2}p\big)q-n(q-1)+2\Big)p-n(p-1)+1> n-\frac{n-1}{2}p\ ,$$
that is the region 
\beeq
\label{811}
\frac{2+p^{-1}}{pq-1}>\frac{n-1}{2}\ .
\eneq
And the lower bound of $G$ in \eqref{die70} could be improved by \eqref{die80} if $$\Big(\big(n+1-\frac{n-1}{2}q\big)p-n(p-1)+1\Big)q-n(q-1)+2> n+1-\frac{n-1}{2}q\ ,$$
that is the region 
\beeq
\label{812}
\frac{p+1+q^{-1}}{pq-1}>\frac{n-1}{2}\ .
\eneq
Combine \eqref{811} and \eqref{812}, we could obtain the finite time blow up result on the range 
$$\max\{\frac{2+p^{-1}}{pq-1}\ , \frac{p+1+q^{-1}}{pq-1}\}-\frac{n-1}{2}>0\ .$$
 In the following, we only give the the upper bound lifespan estimate when 
 $$\frac{p+1+q^{-1}}{pq-1}-\frac{n-1}{2}>0\ ,$$
since the other case follow the same way. In the same way as in the proof of Theorems \ref{thm4}-\ref{thm2}, we have that there exists $\ep_0>0$, such that for any $\ep\in(0, \ep_0)$, we could assume the lower bound of $F$, $G$ by some sequences
$$F\geq A_j(t+1)^{a_j}\ , G\geq B_j(t+1)^{b_j}\ , j\in\N\ .$$
We start the iteration with 
\beeq 
\label{k4}
B_1=\ep^q, \ b_1=n+1-\frac{n-1}{2}q\ ,
\eneq and by \eqref{die80} we have the relation
 \begin{align}
 \label{k1}
\begin{cases}
A_j =\frac{c_0B^{p}_{j-1}}{a_j}\ , a_j =pb_{j-1}-n(p-1)+1\ , j\geq 2\\
\\
 B_{j}=\frac{c_1A^{q}_j}{(b_j-1)b_j}\ , b_{j}=q a_j-n(q-1)+2\ , j\geq 2
\end{cases}
\end{align}
Then we get that
$$
b_j=\Big( \frac{p+1+q^{-1}}{pq-1}-\frac{n-1}{2}\Big)q(pq)^{j-1}+n-\frac{q+2}{pq-1}\ .$$
 For the coefficient $B_j$, it is easy to see that
$$
B_j=\frac{c_1c^{q}_0B_{j-1}^{pq}}{\big(a_j\big)^{q}(b_j-1)b_j}\geq\frac{c_1c^{q}_0B_{j-1}^{pq}}{\big(p b_j+1\big)^{q+2}}\geq \frac{MB_{j-1}^{pq}}{(pq)^{j(q+2)}}
$$
where
$$M=c_1c^{q}_0\Big( \frac{p+1+q^{-1}}{pq-1}+\frac{q+2}{pq-1}+n+1\Big)^{-(q+2)}\ .$$
Hence we get 
\begin{align*}
\ln B_j &\geq (pq)\ln B_{j-1} +\ln M-j(q+2)\ln (pq)\\
&\geq (pq)^2\ln B_{j-2}+(pq+1)\ln M-\big(pq(j-1)+j\big)(q+2)\ln (pq)\\
&\geq (pq)^{j-1}\ln B_1-(pq)^{j-1}\sum^{j-1}_{k=1}\frac{(k+1)(q+2)\ln (pq)-\ln M}{(pq)^{k}}
\end{align*}
Note that the series is convergence, that is
$$S(\infty)=\sum^{\infty}_{k=1}\frac{(k+1)(q+2)\ln (pq)+|\ln M|}{(pq)^{k}}<\infty\ .$$
Then we have 
$$\ln B_j\geq (pq)^{j-1}(\ln B_1-S(\infty))\ .$$
Thus 
\begin{align*}
G&\geq B_{j}(t+1)^{b_j}\\
&\geq e^{(pq)^{j-1}\Big(\ln B_1-S(\infty)+ \big(\frac{p+1+q^{-1}}{pq-1}-\frac{n-1}{2}\big)q\ln(t+1)\Big)}e^{(n-\frac{q+2}{pq-1})\ln (t+1)}\ , \forall j\in \N\ ,
\end{align*}
which yields the lifespan estimate
$$T_{\ep} \les\ \ep^{-\big( \frac{p+1+q^{-1}}{pq-1}-\frac{n-1}{2}\big)^{-1}}\ \ .$$ 
Similarly, if we assume 
$$\frac{2+p^{-1}}{pq-1}-\frac{n-1}{2}>0\ ,$$
with the iteration relation
\begin{align}
\label{k2}
\begin{cases}
A_j =\frac{c_0B^{p}_{j}}{a_j}\ , a_j =pb_{j}-n(p-1)+1\ , j\geq 2\\
\\
 B_{j}=\frac{c_1A^{q}_{j-1}}{(b_j-1)b_j}\ , b_{j}=q a_{j-1}-n(q-1)+2\ , j\geq 2
\end{cases}
\end{align}
and 
\beeq
\label{k3}
A_1=\ep^{p}\ , \ a_1=n-\frac{n-1}{2}p\ ,
\eneq
we will have that 
$$T_\ep \les\ \ep^{-\big(\frac{2+p^{-1}}{pq-1}-\frac{n-1}{2}\big)^{-1}}\ .$$
Hence, we obtain the lifespan estimate
$$T_\ep \les \min\{\ep^{-\big(\frac{p+1+q^{-1}}{pq-1}-\frac{n-1}{2}\big)^{-1}}\ , \ \ep^{-\big(\frac{2+p^{-1}}{pq-1}-\frac{n-1}{2}\big)^{-1}}\}\ .$$

\subsection{Blow up for $
\Ga_{GG}^{*}(p, q, n)>0, c\neq 1$ and $\Ga_{GG}(p, q, n)>0, c=1$.}

We define functionals $$F(t)=\int_{D_2} \big(u_t\psi_1+\la cu\psi_1\big)dv_{\gm}\ , H(t)=\int_{D_2}  v\psi_2 dv_{\gm}\ .$$
Then by the system \eqref{1}, we have 
\beeq\label{81300}
F'(t)=\int_{D_2} |v|^{p}\psi_1 dv_{\gm}\ , \ H''(t)+2\la H'=\int_{D_2} |u_t|^{q}\psi_2 dv_{\gm}\ .
\eneq
In order to relate $F'$ with $H$, we let $H_1(t)=\int_{D_2} u_t \psi_1 d v_{\gm}$ and it is easy to check that 
$$(2H_1-F)'+2c\la (2H_1-F)=\int_{D_2} |v|^{p}\psi_1 d v_{\gm}\geq 0\ , $$
thus by the assumption on the initial data , we have 
$$(2H_1-F)(t)\geq (2H_1-F)(0)\geq 0\ ,$$ 
which yields 
$$\frac{F}{2}\leq H_1(t)=\int_{D_2} u_t\psi_1dv_{\gm}\ .$$
By applying H\"older's inequality to $H_1$ and $H$, we have
\begin{align*}
\int_{D_{2}} |v|^{p}\psi_1dv_{\gm} \gtrsim \frac{\Big(\int_{D_2} v\psi_2 dv_{\gm}\Big)^{p}}{\Big(\int_{D_2}  \psi_{1}^{-\frac{p'}{p}}\psi^{p'}_2 dv_{\gm}\Big)^{p/p'}} \gtrsim\frac{H^p}{\Big(\int_{D_2}  \psi_{1}^{-\frac{p'}{p}}\psi^{p'}_2 dv_{\gm}\Big)^{p/p'}}\ , 
\end{align*}
\begin{align*}
\int_{D_{1}} |u_t|^{q}\psi_2dv_{\gm} \gtrsim \frac{\Big(\int_{D_1} u_t\psi_1 dv_{\gm}\Big)^{q}}{\Big(\int_{D_1}  \psi_{2}^{-\frac{q'}{q}}\psi^{q'}_1 dv_{\gm}\Big)^{q/q'}}\gtrsim 
\frac{H_{1}^q}{\Big(\int_{D_1}  \psi_{2}^{-\frac{q'}{q}}\psi^{q'}_1 dv_{\gm}\Big)^{q/q'}} \ .
\end{align*}
 By \eqref{81300} and Lemma \ref{le1}, we get the following ODE system 
\beeq
\label{die410}
\begin{cases}
F'(t) \geq \frac{c_0|H|^p}{e^{\la (c-1)t}(t+1)^{\frac{(n-1)(p-1)}{2}}}\ , \ F(0)\geq c_0\ep\ , \\
\\
H''+2\la H' \geq \frac{c_1|F|^q}{e^{\la (1-c)t}(t+1)^{\frac{(n-1)(q-1)}{2}}}\ , \ H(0)\geq c_1\ep \ ,\\
\end{cases}
\eneq
for some constants $c_0, c_1>0$.

Note that the above system could be reduced to system \eqref{die4108}. In fact, we have
$$(e^{2\la t}H')'=e^{2\la t}(H''(t)+2\la H')\geq \frac{c_1e^{2\la t}F^q}{e^{\la (1-c)t}(t+1)^{(n-1)(q-1)/2}}\ ,$$
thus
\begin{align*}
H'\geq &e^{-2\la t}\int^{t}_{t/2}\frac{c_1e^{2\la \tau}F^q(\tau)}{e^{\la (1-c)\tau}(\tau+1)^{(n-1)(q-1)/2}}d\tau+H'(0)\\
\geq &\frac{c_1(F(t/2)^q}{e^{\la (1-c)t}(t+1)^{(n-1)(q-1)/2}}\big(\frac{1-e^{-\la t}}{2\la}\big)\\
\geq &\frac{c_2(F(t/2)^q}{e^{\la (1-c)t}(t+1)^{(n-1)(q-1)/2}}\ , t\geq 2\ ,
\end{align*}
for some constant $c_2>0$ depends on $c_1$ and $\la$.  Hence we get the blow up result when $\Ga_{GG}^{*}(p, q, n)>0, c\neq 1$ and $\Ga_{GG}(p, q, n)>0, c=1$ as well as the  corresponding lifespan estimates.

\bibliographystyle{plain1}

%\end{CJK}
\end{document}